\documentclass[english,12pt]{article}
\usepackage[T1]{fontenc}
\usepackage[utf8]{inputenc}
\usepackage{lmodern}
\usepackage[a4paper]{geometry}
\usepackage{babel}
\usepackage{amssymb}
\usepackage{amsmath}
\usepackage{amsthm}
\usepackage{esint}
\usepackage{mathrsfs}
\usepackage{graphicx}
\usepackage{pifont}
\usepackage{amsfonts}
\usepackage{comment}
\usepackage{tikz}
\usepackage[backend=bibtex,style=numeric]{biblatex}
\usepackage{esint}
\usepackage{dsfont}
\usepackage{subcaption}
\usepackage{wasysym}
\usepackage{bbm}
\usepackage[all]{xy}
\usepackage{hyperref}
\usepackage{comment}
\theoremstyle{theorem}
\newtheorem{theorem}{Theorem}[section]
\newtheorem{lemma}[theorem]{Lemma}
\newtheorem{definition}[theorem]{Definition}
\newtheorem{proposition}[theorem]{Proposition}
\newtheorem{corollaire}[theorem]{Corollary}

\usepackage{authblk}

\title{The Hausdorff measure of the boundary of the Brownian disk}
\author[1]{Alexis Metz-Donnadieu \\ \small{\href{mailto:ametzdonnadieu@dma.ens.fr}{ametzdonnadieu@dma.ens.fr}}}
\affil[1]{Institut de mathématiques d’Orsay, Faculté des sciences d’Orsay, Université Paris-Saclay, 91405 Orsay Cedex, France.
	D\'epartement Math\'emathiques et Applications, \'Ecole Normale Sup\'erieure, 75005 Paris, Université PSL, France.}

\date{}                     %% if you don't need date to appear
\begin{document}
	\maketitle
	\begin{abstract}
Consider the boundary $\partial \mathbb D$ of the Brownian disk $\mathbb D$ as a metric space by endowing it with the (restriction of the) metric of $\mathbb D$. We show that the uniform measure on $\partial \mathbb D$ coincides with the Hausdorff measure associated with the gauge function $h(s)=\kappa s^2\log\log(1/s)$ for some deterministic constant $\kappa>0$. We also state the analogous result for the boundary of the Brownian half-plane $\mathbb H$. This proves in	particular that the uniform measure on the boundary of the Brownian disk (resp. the Brownian half-plane) is determined by the metric on $\mathbb D$ (resp. on $\mathbb H$).
	\end{abstract}
	
	\section*{Introduction}
The free Brownian disk is a basic model of random geometry that has been extensively studied in the past years. It appears as the scaling limit of Boltzmann distributed random maps with a boundary when the boundary size goes to $+\infty$ (see in particular \cite{1, 2, 3, 10}). The limit holds for the so-called Gromov--Hausdorff--Prohorov topology. The Brownian disk is a random compact measured metric space and the first construction of (a pointed version of) this object was given by Bettinelli and Miermont in \cite{3}. A similar construction of an unpointed version of this object "seen from a boundary point" was given by Le Gall in \cite{15}. The (pointed or unpointed) free Brownian disk $\mathbb D$ is homeomorphic to the closed unit disk in the plane so that one can consider its boundary $\partial \mathbb D$. Moreover in all the standard constructions of the Brownian disk, the boundary $\partial \mathbb D$ is naturally identified with the circle $\mathbb S^1$ and there is therefore a natural notion of a uniform measure on the boundary, given by the pushforward of Lebesgue measure on $\mathbb S^1$ under this identification. In the approximation of the Brownian disk by random planar maps, the uniform measure on the boundary appears as the scaling limit of the counting measure. This indeed follows from the convergence in the GHPU topology proved in \cite{1} for triangulations and in \cite{10} for quadrangulations. It was also shown in \cite{12}  that the uniform measure on the boundary is the limit of the renormalised uniform measure on small tubular neighbourhoods of the boundary. This measure is our main object of interest in this paper. 

In \cite{16}, Le Gall showed that the metric determines the uniform measure on the Brownian sphere. More precisely, the uniform volume measure on the Brownian sphere is a Hausdorff measure for an explicit gauge function (known up to a multiplicative constant). It is therefore natural to ask if such a result is also true for the uniform measure on the boundary of the Brownian disk. The aim of this paper is to answer this question, and we will show that the uniform measure on the boundary is almost surely a multiple of the Hausdorff measure associated with the gauge function $s^2\log\log(s^{-1})$.

\begin{theorem}\label{MainResult}
	There exists a deterministic constant $\kappa\in (0, \infty)$ such that almost surely the uniform measure on the boundary $\partial \mathbb D$ of the free Brownian disk coincide with $\kappa m^h$, where $m^h$ is the Hausdorff measure on $\partial \mathbb D$ associated with the metric on $\mathbb D$ and the gauge function $h(s)= s^2\log\log\frac{1}{s}$. 
\end{theorem}

The Brownian half-plane $\mathbb H$ is another important model of Brownian geometry which was introduced by Caraceni and Curien in \cite{5} and studied in greater detail by Gwynne and Miller in \cite{11} and by Baur, Miermont and Ray in \cite{2}. The Brownian half-plane can be viewed as the tangent cone in distribution to the Brownian disk at a distinguished point of the boundary (see Chapter $8$ of \cite{4} for the definition of tangent cones to metric spaces), and it also appears as the scaling limit of the uniform infinite half-planar quadrangulation as it was shown in \cite{11}. Similarly as for the Brownian disk, one can defined a uniform measure on the boundary $\partial \mathbb H$ of the Brownian half-plane. Moreover from the coupling described in Section $4.2$ of \cite{11} we know that we can couple the Brownian disk and the Brownian half-plane in such a way that small enough balls centred at the distinguished point on the boundary of the Brownian disk, respectively on the boundary of the Brownian half-plane, are isometric. Therefore, the result of Theorem \ref{MainResult} easily yields an analogous statement for the boundary of the Brownian half-plane.

\begin{corollaire}
	With the same constant $\kappa$ as in Theorem \ref{MainResult}, the uniform measure on $\partial \mathbb H$ coincides with $\kappa$ times the Hausdorff measure associated with the gauge function $h$.
\end{corollaire}

The paper is organized as follows, section $2$ gives some general preliminaries about the Brownian disk and in particular recalls the definitions of the random measures that we consider in this paper. In section $3$, we give some estimates on Bessel processes that play a central role in the proof of the Theorem \ref{MainResult}. Finally section $4$ gives the proof of Theorem \ref{MainResult}.
	
	\section{Preliminaries}
	
This section is dedicated to recalling the constructions of the objects and concepts that we consider in this paper. We start by recalling the basic construction and properties of the Brownian snake in section 1.1. In sections 1.2 and 1.3, we give the construction of the Brownian disk and the Brownian half-plane seen from a boundary point as given in \cite{15}. Section 2.4 provides some general preliminary results about Hausdorff measures. 
	\subsection{The Brownian snake} 
Let $\mathcal W_0$ be the set of all finite paths in $\mathbb R$ started at $0$, namely the set of all continuous functions $w:[0, \zeta(w)]\to \mathbb R$, where $\zeta(w)$ is a non-negative real number called the lifetime of $w$, and such that $w(0)=0$. This set is endowed with the distance:
\begin{equation*}
	d_{\mathcal W_0}(w, w')=|\zeta(w)-\zeta(w')|+\sup_{t\geq 0} |w(t\wedge \zeta(w))-w'(t\wedge \zeta(w'))|.
\end{equation*}
\begin{definition}
	A snake trajectory $\omega$ starting at $0$ is a continuous function $s\mapsto \omega_s$ from $\mathbb R_+$ into $\mathcal W_0$, with $\zeta(\omega_0)=0$ and such that:
	\begin{itemize}
		\item[\textbullet] the lifetime $\sigma(\omega):=\sup\{s\geq 0, \ \zeta(\omega_s)> 0\}$ of $\omega$ is finite,
		\item[\textbullet] for all $0\leq s\leq s'$, $\omega_s(t)=\omega_{s'}(t)$ whenever $t\leq \min_{u\in [s, s']}\zeta(\omega_u)$.
	\end{itemize}
	We will write $\mathcal S_0$ for the set of all snake trajectories starting at $0$, endowed with the distance:
	\begin{equation*}
		d_{\mathcal S_0}(\omega, \omega')=|\sigma(\omega)-\sigma(\omega')|+\sup_{s\geq 0}d_{\mathcal W_0}(\omega_s,\omega'_s).
	\end{equation*}
\end{definition}
For $\omega\in \mathcal S_0$ and $s\geq 0$, write $\zeta_\omega(s)= \zeta(\omega_s)$ and $\hat \omega(s):= \omega_s(\zeta_\omega(s))$. One easily shows that $\omega$ is entirely specified given the two functions $\zeta_\omega:\mathbb R_+\to \mathbb R_+$ and $\hat \omega:\mathbb R_+\to \mathbb R$ which are respectively called the \emph{lifetime} and \emph{snake head} functions associated to $\omega$. In the following, we will also write $W_*(\omega)$ for the minimal label taken by the snake trajectory $\omega$:
\begin{equation*}
	W_*(\omega)=\min_{s\geq 0}\min_{t\in [0, \zeta_\omega(s)]} \omega_s(t)=\min_{s\in [0, \sigma(\omega)]}\hat \omega(s).
\end{equation*}
To each snake trajectory $\omega$ one can associate a labelled $\mathbb R$-tree $T_\omega$ called the \emph{genealogical tree of trajectories of $\omega$}. It is the labelled metric space obtained by considering the pseudo-distance $d_\omega$ on $[0, \sigma(\omega)]$ defined for all $s, t\in [0, \sigma(\omega)]$ by:
\begin{equation*}
	d_\omega(s, t)= \zeta_\omega(s)+\zeta_\omega(t)-2\min_{u\in [s, t]}\zeta_\omega(u),
\end{equation*}
and then constructing $T_\omega$ as the quotient space of $[0, \sigma(\omega)]$ for the equivalence relation $s\sim t$ iff $d_\omega(s, t) =0$ and endowing this quotient with the metric induced by $d_\omega$. 
Let $p_{\omega}:[0, \sigma(\omega)]\to T_\omega$ be the canonical projection and let $\rho_\omega:=p_\omega(0)$ be a distinguished point in $T_\omega$ that we call the \emph{root} of $T_\omega$. Two points having the same image by $p_\omega$ also have the same image by $\hat \omega$, and we therefore have a natural labelling $\ell_\omega:T_\omega\to \mathbb R$ induced by $\hat\omega$ and characterized by $\hat \omega =\ell_\omega\circ p_\omega$.
\begin{definition}
	The Brownian snake excursion measure is the unique measure $\mathbb N_0$ on $\mathcal S_0$ such that the pushforward of $\mathbb N_0$ under the function $\omega\mapsto \zeta_\omega$ is the Itô measure of positive Brownian excursions, normalized so that \[
	\mathbb N_0\left(\sup_{s\geq 0}\zeta_\omega(s)\geq \varepsilon\right)=\frac{1}{2\varepsilon},\]
	and such that conditionally on $\zeta_\omega$, the function $\hat \omega:\mathbb R_+\to \mathbb R$ is a centred Gaussian process with continuous sample paths and covariance kernel $K(s, s ')=\min_{u\in[s, s']}\zeta_\omega(u)$.
\end{definition}
	\subsection{The free Brownian disk seen from the boundary}
In this section, we recall the construction given in \cite{15} of the free Brownian disk of perimeter $1$ seen from a boundary point. Let $(R_t)_{t\in [0, 1]}$ be a $5$-dimensional Bessel bridge of length $1$ starting and ending at $0$ and, conditionally on $R$, let $\mathcal N_{\mathbb D}=\sum_{i\in I}\delta_{(t_i, \omega^i)}$ be a Poisson point measure on $[0, 1]\times \mathcal S$ of intensity $2\mathbbm 1_{\{W_*(\omega)\geq -\sqrt 3 R_t\}}\mathrm dt\mathbb N_{0}(\mathrm d\omega)$.
Consider the set:
\begin{equation}
	\label{ID}
	\mathcal I:= \left([0, 1]\cup \bigcup_{i\in I} T_{\omega^i}\right)\Big/\sim,
\end{equation} where $\sim$ is the equivalence relation identifying $\rho_{\omega^i}$ and $t_i$ for all $i\in I$ (and no other pair of points). We endow this set with the maximal distance $d_{\mathcal I}$ whose respective restrictions to $T_{\omega^i}$ (for every $i\in I$) and to $[0, 1]$ coincide with $d_{\omega^i}$ and with the standard distance on $[0, 1]$ respectively. The distance between two points $x\in T_{\omega^i}$ and $y\in T_{\omega^j}$, where $i\neq j$ is more explicitly given by 
\[d_{\mathcal I}(x, y):= d_{\omega^i}(x,\rho_{\omega^i})+|t_i-t_j|+ d_{\omega^j}(y, \rho_{\omega^ j}).\]
The metric space $(\mathcal I, d_{\mathcal I})$ is compact, and we define a labelling $\ell:\mathcal I\to \mathbb R$ by:
\begin{equation}
	\label{labelFunct}
	\ell(x)=\left\{\begin{array}{l}
		\ell_{\omega^i}(x)+\sqrt 3 R_{t_i} \ \ \ \ \text{ if $x\in T_{\omega^i}$, for $i\in I$,}\\
		\sqrt 3R_x \ \ \ \ \ \ \ \ \ \ \ \ \ \ \ \ \text{ if $x\in [0, 1]$.}
	\end{array}\right.
\end{equation}
One can show that the quantity $\Sigma:=\sum_{i\in I}\sigma(\omega^i)$ is almost surely finite, and it is possible to concatenate the functions $p_{\omega^i}$ in the order prescribed by the points $(t_i)_i$ to obtain a projection $\pi:[0, \Sigma]\to \mathcal I$. Formally, to define $\pi$, let $\mu=\sum_{i\in I}\sigma(\omega^i)\delta_{t_i}$ be the point measure on $[0, 1]$ giving weight $ \sigma(\omega^i)$ to $t_i$, for every $i\in I$, and consider the left-continuous inverse $\mu^{-1}$ of its cumulative distribution function, given for every $s\in [0, \Sigma]$ by:
\begin{equation*}
	\mu^{-1}(s):=\inf\{t \ : \ \mu[0, t]\geq s\}.
\end{equation*}
For $s\in [0, \Sigma]$ we then simply set $\pi(s)=p_{\omega^{i}}(s-\mu[0, t_i))$ if $\mu^{-1}(s)=t_i$ for some $i\in I$ and $\pi(s)=\mu^{-1}(s)$ otherwise.
	
This projection $\pi$ allows us to consider intervals on $\mathcal I$. For $u, v\in \mathcal I$, there exists a minimal interval $[s, t]$ in $[0, \Sigma]$ such that $\pi(s)=u$ and $\pi (t)=v$, where by convention $[s, t]=[s, \Sigma]\cup[0, t]$ if $s> t$. We then write $[|u, v|]$ for the subset of $\mathcal I$ defined by $\{\pi(b), \ b\in [s, t]\}$ (note that we have $[| u, v|]\neq [|v, u|]$ in general). We set:
\begin{equation}
	\label{distance0D}
	\forall u, v\in \mathcal I, \ \ D^\circ(u, v):=\ell(u)+\ell(v)-2\max\left(\min_{w\in [|u, v|]}\ell(w), \min_{w\in [|v, u|]}\ell(w)\right),
\end{equation} 
and the pseudo metric $D$ is defined for all $u, v\in \mathcal I$ by:
\begin{equation}
	\label{distanceD}
	D(u, v):= \inf_{u=u_0,u_1, \cdots, u_p=v}\sum_{j=0}^{p-1}D^\circ(u_j, u_{j+1}),
\end{equation}
where the infimum is taken over all choices of the integer $p\geq 1$ and elements $u_0,\cdots, u_p\in \mathcal I$ such that $u_0= u$ and $u_p=v$. By taking the quotient of $\mathcal I$ for the equivalence relation $s\sim t\Leftrightarrow D(s, t)=0$ we get a compact metric set $\mathbb D$ endowed with the distance induced by $D$ (in the following we will also write $D$ for the metric induced on $\mathbb D$, with a slight abuse of notation). We write $\Pi:\mathcal I\to \mathbb D$ for the canonical projection and $\Gamma:[0, 1]\to \mathbb D$ for the restriction of $\Pi$ to the segment $[0, 1]$ viewed as a subset of $\mathcal I$ as in (\ref{ID}). Note that $\Gamma(0)=\Gamma(1)$ since it is clear from our definitions that $D(0, 1)=0$. 
\begin{definition}
	The free Brownian disk of perimeter $1$ is the (random) curve-decorated compact measured metric space $(\mathbb D, D, \mathbf V, \Gamma)$ where $(\mathbb D, D)$ and $\Gamma$ are defined as above and where $\mathbf V$ is a Borel measure on $\mathbb D$ called the volume measure and defined as the pushforward of the Lebesgue measure on $[0, \Sigma]$ under the mapping $\Pi\circ \pi:[0, \Sigma]\to \mathbb D$. The subset of $\mathbb D$ defined by $\partial \mathbb D:=\Gamma([0, 1])$ is called the boundary of $\mathbb D$. It is identified with $[0, 1]/_{0\sim 1}$ through the mapping $\Gamma$.
\end{definition}
The space of all curve-decorated compact measured metric spaces (modulo isometries) is equipped with the Gromov--Hausdorff--Prohorov uniform metric (GHPU for short). With the exception of the proof of Lemma \ref{rerooting} we will not need the precise definition and properties of this metric and we refer to  \cite{11} for a detailed discussion of curve-decorated compact measured metric spaces and of the GHPU topology.

For any two points $u,v\in \mathcal I$ such that $\Pi(u)=\Pi(v)$ we have $\ell(u)= \ell(v)$. Thus $\ell$ induces a function $\mathbb D\to \mathbb R$ for which we keep the same notation $\ell$. Crucially one can verify from formulas (\ref{distance0D}) and (\ref{distanceD}) that for all $x\in \mathbb D$ we have 
\begin{equation}
	\label{DistLabeling}\ell(x)=D(\Gamma(0), x).
\end{equation}
One important aspect of this construction that we will use repeatedly in this paper is the fact that the root point $\Gamma(0)$ is "uniform" on $\partial \mathbb D$. Namely, fix some $u\in [0, 1]$. We define a new curve $\Gamma^{[u]}$ in $\mathbb D$ corresponding to $\Gamma$ "re-rooted" at $\Gamma(u)$ by setting for all $t\in [0, 1]$:
\begin{equation*}
	\Gamma^{[u]}(t)=\left\{\begin{array}{l}\Gamma(t+u) \ \ \ \ \ \text{ if } t+u\leq 1 \\ \Gamma(t+u-1) \text{ otherwise}
	\end{array}\right.. 
\end{equation*} 
\begin{lemma}\label{rerooting}
	For every fixed $u\in [0, 1]$, $(\mathbb D, D,  \mathbf V, \Gamma)$ has the same law as $(\mathbb D, D, \mathbf V, \Gamma^{[u]})$, where both $(\mathbb D, D,  \mathbf V, \Gamma)$ and $(\mathbb D, D, \mathbf V, \Gamma^{[u]})$ are seen as random elements of the set of all curve-decorated  compact measured metric spaces endowed with the GHPU topology. 
\end{lemma}
\begin{proof}[Proof of Lemma \ref{rerooting}] We can deduce this statement from the characterization of the Brownian disk as the scaling limit of Boltzmann triangulations for the GHPU topology. We do not present the details in full length and just sketch the main elements of the proof (see Lemma 17 in \cite{3} and Theorem 8.1 in \cite{13} for similar results and more detailed proofs). For every integer $L\geq 1$, let $\mathcal T_L$ be a Boltzmann distributed rooted triangulation with a simple boundary $\partial \mathcal T_L$ of size $L$ (we refer to \cite{1} for a precise definition of this object). Let $\Delta_L$ stand for the graph distance on the vertex set $V(\mathcal T_L)$ of $\mathcal T_L$ and write $d_L= \frac{3}{2}L^{-1/2}\Delta_L$. We also write $\nu_L$ for the counting measure on the set of inner vertices $V(\mathcal T_L)\setminus \partial \mathcal T_L$ scaled by the factor $\frac{3}{4}L^{-2}$. We finally consider the “boundary path” $\Theta_L=(\Theta_L(k))_{0\leq k\leq L}$, which is obtained by letting $\Theta_L(0)=\Theta_L(L)$ be the root vertex of $\mathcal T_L$ and $\Theta_L(1), \cdots, \Theta_L(L-1)$ be the points of $\partial \mathcal T_L$ enumerated in clockwise order from $\Theta_L(0)$. We also set $\hat \Theta_L(t)=\Theta_L(\lfloor Lt\rfloor )$ for $t\in [0, 1]$. According to Theorem 1.1 of \cite{1}:
	\begin{equation}
		\label{ConvergenceTriangulations}
		(V(\mathcal T_L), \Delta_L, \nu_L, \hat \Theta_L) \xrightarrow[L\to\infty]{} (\mathbb D, D, \mathbf V, \Gamma),
	\end{equation}
	where the convergence holds in distribution for the GHPU topology (there is a slight technical difficulty related to the fact that $\hat \Theta_L$ is not continuous but this can be easily fixed, see Remark $1.2$ in \cite{10}).
	Fix $u\in [0, 1]$ and let $u_L=L^{-1}\lfloor uL\rfloor$. It is clear from the definition of a Boltzmann triangulation that $(\mathcal T_L, \Delta_L, \nu_L,  \hat \Theta_L )$ has the same law as $(\mathcal T_L, \Delta_L, \nu_L,  \hat\Theta_L^{(u_L)})$ where we wrote for every $t\in [0, 1]$: \[\hat \Theta_L^{(u_L)}(t)=\hat \Theta_L(u_L+t-\lfloor u_L+t\rfloor).\] Then, one easily sees from properties of the GHPU topology that (\ref{ConvergenceTriangulations}) implies that $(\mathcal T_L, \Delta_L, \nu_L,  \hat \Theta_L^{(u_L)})$ also converges in distribution when $L\to \infty$ towards $(\mathbb D, D, \mathbf V,  \Gamma^{[u]})$. By identifying the limiting laws, this imply in particular that $(\mathbb D, D, \mathbf V,\Gamma)$ must have the same law as $(\mathbb D, D, \mathbf V, \Gamma^{[u]})$.
\end{proof}
	\subsection{The Brownian half-plane} 
The Brownian half-plane is an analogue of the Brownian disk with perimeter $1$ discussed in the previous subsection, where the boundary has an infinite perimeter. We present a construction of the Brownian half-plane due to Caraceni and Curien \cite{5} which is very similar to the construction of $\mathbb D$ given above and we just highlight here the main differences. Instead of considering a Bessel bridge, let $(X_t)_{t\in \mathbb R}$ be such that $(X_t)_{t\geq 0}$ and $(X_{-t})_{t\geq 0}$ are two independent $5$-dimensional Bessel processes started from $0$. Conditionally on $X$, we consider a Poisson point measure $\mathcal N_{\mathbb H}=\sum_{j\in J}\delta_{(t_j, \omega^j)}$ on $\mathbb R\times \mathcal S_0$ with intensity:
\begin{equation*}
	2\mathbbm 1_{\{W_*(\omega)>-\sqrt 3 X_{t_i}\}} \mathrm dt \mathbb N_0(\mathrm d \omega).
\end{equation*}
Analogously to the last section, we consider $\mathcal J=\left(\mathbb R\cup \bigcup_{j\in J} T_{\omega^j}\right)\big/_{\sim}$ where $\sim$ identifies $t_j$ and $\rho_{\omega^j}$ for all $j\in J$ and no other pair of points. We endow $\mathcal J$ with the maximal distance $d_{\mathcal J}$ such that the respective restrictions of $d_{\mathcal J}$ to $\mathbb R$ and to $T_{\omega^j}$ coincide with the usual distance on $\mathbb R$ and with $d_{\omega^j}$ respectively (for all $j\in J$). Let $\mu_\infty=\sum_{j\in J}\sigma(\omega^j) \delta_{t_j}$ and define:
\begin{equation*}
	\begin{aligned}
		&\mu_\infty^{-1}(s)=\inf\{t\geq 0, \mu_\infty[0, t]\geq s\} \ \ \ \ \text{ if } s\geq 0, \ \ \\ \text{ and } \ \ &\mu_\infty^{-1}(s)=\sup\{t\leq 0, \mu_\infty[t, 0]\geq -s\} \ \text{ if }  s\leq 0.
	\end{aligned}
\end{equation*} 
We define a surjection $\pi_\infty:\mathbb R\to \mathcal J$ by setting:
\begin{equation*}
	\pi_\infty(s)=\left\{\begin{array}{l}
		p_{\omega^j}\left(s-\mu_\infty[0, t_j)\right)\ \ \ \ \ \ \ \ \ \ \ \text{ if $s\geq 0$ and $\mu^{-1}_\infty(s)=t_j$ for some $j\in J$},\\
		p_{\omega^j}(\sigma(\omega^j)+s+\mu_\infty(t_j,0]) \ \text{ if $s\leq 0$ and $\mu^{-1}_\infty(s)=t_j$ for some $j\in J$},\\
		\mu_\infty^{-1}(s) \ \ \ \ \ \ \ \ \ \ \ \ \ \ \ \ \ \ \ \ \ \ \ \ \ \ \text{ otherwise}.
	\end{array}\right.
\end{equation*} Then again, for all $u, v\in \mathcal J$ we can define an interval $[|u, v|]$ on $\mathcal J$ as the minimal set of the form $\{\pi_\infty(b), \ \ b \in [s, t]\}$ where $\pi_\infty(s)=u$ and $\pi_\infty(t)=v$ and where by convention $[s, t]=(-\infty,t]\cup [s, \infty)$ if $s>t$. Analogously to the definition (\ref{labelFunct}) there is a natural labelling on $\mathcal J$ defined for all $x\in \mathcal J$ by:
\begin{equation*}
	\ell_\infty(x):=	\left\{\begin{array}{l}
		\ell_{\omega^j}(x)+\sqrt 3X_{t_j} \text{ if } x\in T_{\omega^j} \text{ for some } j\in J,\\
		\sqrt 3 X_x \ \ \ \ \ \ \ \ \ \ \ \ \ \ \ \ \ \text{ if } x\in \mathbb R.
	\end{array}\right.
\end{equation*}
This allows us to define analogues of definitions (\ref{distance0D}) and (\ref{distanceD}). For all $u, v\in \mathcal J$, we set:
\begin{align}
	\label{Dinfty0}
	&D^\circ_\infty(u, v):=\ell_\infty(u)+\ell_\infty(v)-2\max\left(\min_{w\in [|u, v|]}\ell_\infty(w), \min_{w\in[|v,u|]}\ell_\infty(w)\right),\\
	\label{Dinfty} &D_\infty(u, v):= \inf_{u=u_0,u_1, \cdots, u_p=v}\sum_{j=0}^{p-1}D^\circ_\infty(u_j, u_{j+1}),
\end{align}
where the infimum is taken over all choices of the integer $p\geq 1$ and elements $u_0,\cdots, u_p\in \mathcal J$ such that $u_0=u$ and $u_p=v$.This allows us to construct a random (non-compact) metric space $(\mathbb H, D_{\infty})$ by taking the quotient space of $\mathcal J$ for the pseudo-distance $D_\infty$. Note that for all $u,v\in \mathbb H$, we get from formulas (\ref{Dinfty0}) and (\ref{Dinfty}) that
\begin{equation}
	\label{boundD}
	|\ell_\infty(u)-\ell_\infty(v)|\leq D_\infty(u,v)\leq D^\circ_\infty(u,v) \leq \ell_\infty(u)+\ell_\infty(v).
\end{equation}
Let us write $\Pi_\infty: \mathcal J\to \mathbb H$ for the canonical projection associated with this quotient and let $\Gamma_\infty:\mathbb R\to \mathbb H$ be the restriction of $\Pi_\infty$ to $\mathbb R$ viewed as a subset of $\mathcal J$. This defines an (infinite) curve on $\mathbb H$ whose range $\partial \mathbb H=\Gamma_\infty(\mathbb R)$ is called the boundary of the Brownian half-plane. Letting $\mathbf V_\infty=(\Pi_\infty\circ \pi_\infty)_*\lambda_{\mathbb R}$ finally allows us to define a volume measure on $\mathbb H$, where we wrote $\lambda_{\mathbb R}$ for Lebesgue measure on $\mathbb R$.
The random curve-decorated measured metric space $(\mathbb H, D_\infty, \mathbf V_\infty, \Gamma_\infty)$ is called the \emph{Brownian half-plane}. Note that $\partial \mathbb H$ and $\mathbb H$ are not compact, but one can define a topology on the set of all curve-decorated locally compact measured metric spaces that extends the GHPU topology. This is provided by the \emph{local} GHPU metric and then again we refer to \cite{11} for a proper definition of this metric.

Let us finish this section by recording a useful lemma stating a zero--one law for events that only depend on arbitrarily small neighborhoods of $0$ in $\mathcal J$. For $\delta>0$, write $\mathcal N_{\mathbb H}^\delta$ for the restriction of $\mathcal N_{\mathbb H}$ to $[-\delta, \delta]\times \mathcal S_0$.

\begin{lemma}[zero--one law for the germ algebra in $\mathbb H$] \label{Trivialalgebra} The $\sigma$-algebra:
	\begin{equation*}
		\mathcal F_{0+}:=\bigcap_{r>0} \sigma\left((X_s)_{s\in [-\delta, \delta]}, \mathcal N_{\mathbb H}^\delta\right),
	\end{equation*}  is trivial. Namely, any event that is measurable with respect to  $\mathcal F_\delta:=\sigma\left((X_s)_{s\in [-\delta, \delta]}, \mathcal N_{\mathbb H}^\delta\right)$, for every $\delta>0$ has probability $0$ or $1$.
\end{lemma}
\begin{proof}[Proof of Lemma \ref{Trivialalgebra}]
	Let $E$ be an event in $\mathcal F_{0+}$. Fix $\eta>0$ and consider a random variable $Z$ of the form:
	\begin{equation}
		\label{Zvariable}
		Z=\Phi\Big(X_{\eta} ,X_{-\eta}, X_{t_1}, \cdots, X_{t_k}, \mathcal N_{\mathbb H}(B_1), \cdots, \mathcal N_{\mathbb H}(B_l)\Big),
	\end{equation} 
	for real numbers $t_1, \cdots, t_k\in \mathbb R\setminus [-\eta, \eta]$, measurable subsets $B_1, \cdots, B_l$ of $(\mathbb R\setminus [-\eta, \eta])\times \mathcal S_0$ and a bounded continuous function $\Phi:\mathbb R^{l+k+2}\to \mathbb R$. By the Markov property for Bessel processes and properties of Poisson point processes, conditionally on $X_\eta$ and $X_{-\eta}$, the $\sigma$-algebra $\mathcal F_\eta$ is independent of $(X_{t_1}, \cdots, X_{t_k},\mathcal N_{\mathbb H}(B_1), \cdots, \mathcal N_{\mathbb H}(B_l))$. Since $E$ is $\mathcal F_\eta$-measurable, we get:
	\begin{equation*}
		\mathbb E[\mathbbm 1_E Z]=\mathbb E[\mathbbm 1_E \mathbb E[Z \ | \ X_{-\eta}, X_\eta]].
	\end{equation*}
	Now by a direct consequence of the Blumenthal zero--one law for right continuous Feller processes the random variables $X_\eta$ and $X_{-\eta}$ are independent of $\mathcal F_{0+}:=\bigcap_{\delta>0} \mathcal F_{\delta}$ and it follows that: \[\mathbb E[\mathbbm 1_E Z]=\mathbb E[\mathbbm 1_E]\mathbb E\big[\mathbb E[Z|X_\eta, X_{-\eta}]\big]= \mathbb P(E)\mathbb E[Z].\] Since this is true for every random variables of the form (\ref{Zvariable}) and every events $E$ in $\mathcal F_{0+}$ it follows that $\mathcal F_{0+}$ is independent of $\sigma\left((X_s)_{s\in \mathbb R\setminus [-\eta, \eta]}, \mathcal N_{\mathbb H}|_{(\mathbb R\setminus [-\eta, \eta])\times \mathcal S_0}\right)$ for all $\eta>0$, where we wrote $\mathcal N_{\mathbb H}|_{(\mathbb R\setminus [-\eta, \eta])\times \mathcal S_0}$ for the restriction of $\mathcal N_{\mathbb H}$ to $(\mathbb R \setminus [-\eta, \eta])\times \mathcal S_0$. By letting $\eta$ tend to $0$ we get that $\mathcal F_{0+}$ is independent of $\sigma((X_s)_{s\in \mathbb R}, \mathcal N_{\mathbb H})$,  hence of itself. Each event in $\mathcal F_{0+}$ must have probability $0$ or $1$.
\end{proof}
	\subsection{Hausdorff measures}
Let $(E, \Delta)$ be a compact metric space and fix $\delta>0$. A function $h:[0, \delta] \to \mathbb R_+$ is said to be a gauge function if it is positive, continuous, nondecreasing and satisfies $h(0)=0$ and $h(s)>0$ for all $s>0$. For every $A\subset E$, the Hausdorff measure $m^h(A)\in [0, \infty]$ of $A$ associated with the gauge function $h$ is then defined by:
\begin{equation}
	\label{HausdorffMeasures}
	m^h(A)=\lim\limits_{\varepsilon\downarrow 0}\left(\inf_{(U_i)_{i\in I}\in \text{Cov}_{\varepsilon}(A)}\sum_{i\in I} h(\text{diam}_{\Delta}(U_i))\right),
\end{equation}
where $ \text{Cov}_{\varepsilon}(A)$ is the set of all coverings of $A$ by subsets of $E$ of diameter at most $\epsilon$, and where $\text{diam}_\Delta(U_i)$ denotes the diameter of $U_i$ for the metric $\Delta$. We will use the following result that can be found in \cite{8}, Lemma 2.1. If $x\in E$ and $r>0$, we write $ B_r(x)$ for the closed ball of radius $r$ centred at $x$ in $E$.
\begin{lemma}
	\label{Hausdorffmeasurescontrole}
	Assume that there exists a constant $c>0$ such that the function $h$ satisfies $h(2r)\leq c\cdot h(r)$ for every $r\in (0, \delta/2]$. There exist two positive constants $M_1$ and $M_2$, which only depend on $c$, such that the following holds for every finite Borel measure $\mu$ on $E$, every Borel subset $A$ of $E$ and every $b>0$:
	\begin{itemize}
		\item[(i)] If, for every $x\in A$, we have $
		\limsup_{r\downarrow 0}\frac{\mu(B_r(x))}{h(r)}\leq b$,
		then $m^h(A)\geq M_1b^{-1}\mu(A)$.\\
		\item[(ii)] If, for every $x\in A$, we have $
		\limsup_{r\downarrow 0}\frac{\mu(B_r(x))}{h(r)}\geq b$,
		then $m^h(A)\leq M_2 b^{-1}\mu(A).$
	\end{itemize}
\end{lemma}

In the following, we are interested in the metric spaces associated with the boundaries $\partial \mathbb D$ and $\partial \mathbb H$ of the Brownian disk and the Brownian half plane, endowed with the metric induced by $D$ and $D_\infty$ respectively. As it will become apparent in the following analysis, the correct gauge function to consider in this case is:
\begin{equation}
	\label{defh}
	h(s)=s^2\log\log\frac 1 s.
\end{equation} 
We will write $m_{\mathbb D}^h$ (resp. $m_{\mathbb H}^h$) for the Hausdorff measure associated with the gauge function $h$ on the metric space $(\partial \mathbb D, D|_{\partial \mathbb D})$ (resp. $(\partial \mathbb H, D_\infty|_{\partial \mathbb H})$). Recall that in the constructions of $\mathbb D$ and $\mathbb H$ given above, $\partial \mathbb D$ and $\partial \mathbb H$ are identified with the segment $[0, 1]/_{0\sim 1}$ and with $\mathbb R$ respectively. Somewhat abusively, we will see $m^h_\mathbb D$ and $m^h_\mathbb H$ as measures on $[0, 1]$ and on $\mathbb R$ respectively so that expressions of the form $m^h_\mathbb D([0, \epsilon])$ make sense, for $\epsilon\in (0, 1)$. We record a property that will be useful later, which is a consequence of Lemma \ref{Trivialalgebra}:
\begin{proposition}
	\label{deterBrownPlane}
	There exists a deterministic constant $\kappa \in[0, \infty]$ such that almost surely:
	\begin{equation}
		\label{psiprime}
		\limsup_{\epsilon\downarrow 0}\frac{m^h_{\mathbb D}([0, \epsilon])}{\epsilon}=\limsup_{\epsilon\downarrow 0} \frac{m^h_{\mathbb H}([0, \epsilon])}{\epsilon}=\kappa.
	\end{equation}
\end{proposition}
\begin{proof}[Proof of Proposition \ref{deterBrownPlane}]
	Let us write $\phi^{\mathbb H}=\limsup_{\epsilon\downarrow 0} m^h_{\mathbb H}([0, \epsilon])/\epsilon$ and $\phi^{\mathbb D}=\limsup_{\epsilon\downarrow 0}m^h_\mathbb D([0, \epsilon])/\epsilon$. We will first show that $\phi^{\mathbb H}$ is almost surely constant. To see this, according to the zero--one law of Lemma \ref{Trivialalgebra} we only have to show that $\phi^\mathbb H$ is measurable with respect to $\mathcal F_\eta$ for every $\eta>0$.
	For every $x\in \mathcal J$, we write $\rho(x)=t_j$ if $x\in T_{\omega^j}$ for some $j\in J$ and $\rho(x)=x$ otherwise. Fix $\eta>0$ and consider the subset $\mathcal J^{(\eta)}=\{x\in \mathcal J, |\rho(x)|\leq \eta \}$ of $\mathcal J$, which is obtained by only keeping the trees whose roots fall inside $[-\eta, \eta]$. Consider the following functions on $\mathcal J^{(\eta)}\times \mathcal J^{(\eta)}$, defined for all $u,v\in \mathcal J^{(\eta)}$ by:
	\begin{equation*}
		\begin{aligned}
			D^{\circ, (\eta)}_\infty(u, v)&=\ell_\infty(u)+\ell_\infty(v)-2\max\left(\min_{w\in [|u, v|]\cap \mathcal J^{(\eta)}}\ell_\infty(w), \min_{w\in[|v,u|]\cap \mathcal J^{(\eta)}}\ell_\infty(w)\right), \\
			D_\infty^{(\eta)}(u, v)&= \inf_{\substack{u=u_0,u_1, \cdots, u_p=v\\ |\rho(u_i)|\leq \eta}}\sum_{j=0}^{p-1}D^{\circ, (\eta)}_\infty(u_j, u_{j+1}),
		\end{aligned}
	\end{equation*}	
	where the infimum is taken over all choices of $p\geq 1$ and points $u_0, \cdots, u_p$ in $\mathcal J^{(\eta)}$ such that $u_0=u$ and $u_p=v$. The quantity:
	\begin{equation}
		\label{minlabel}
		\omega_\eta:=	\inf\left\{\ell_\infty(x) : \ x\in \mathcal J, |\rho(x)|>\eta\right\},
	\end{equation}
	is strictly positive a.s. (this is a straightforward consequence of \cite{7}, Lemma 3.3 and of a compactness argument). Let $
	\mathcal B^\eta=	\{x\in \mathcal J, \ \ell_\infty(x)\leq \frac{\omega_\eta}{3}\}$ so that in particular we have $\mathcal B^\eta \subset \mathcal J^{(\eta)}$. We claim that almost surely $D^{(\eta)}_\infty$ coincide with $D_\infty$ on $\mathcal B^\eta$.
	This will imply in particular that $\phi^{\mathbb H}$ is $\mathcal F_\eta$-measurable since to compute $\phi^{\mathbb H}$ one only needs to know the values taken by $D_\infty$ on small enough neighbourhoods of $0$, and the definition of $D_\infty^{(\eta)}$ only involves the pair $((X_t)_{t\in [-\eta, \eta]}, \mathcal N_{\mathbb H}^\eta)$ (note that $\omega_\eta$ is not $\mathcal F_\eta$-measurable but this does not matter here). To get our claim, we first note that if $x, y\in \mathcal J$ are such that $\ell_\infty(x)< \omega_\eta$ and $\ell_\infty(y)< \omega_\eta$ then $x, y$ are in $\mathcal J^{(\eta)}$ and points of labels smaller than $\ell_\infty(x)$ and $\ell_\infty(y)$ all lie in $\mathcal J^{(\eta)}$. This readily implies that we must have $D^{\circ, (\eta)}_\infty(x, y)= D^\circ_\infty(x, y)$. Let $u, v\in \mathcal B^\eta$ and fix $u_0, \cdots, u_p$ in $\mathcal J$ such that $u_0=u$ and $u_p=v$. Suppose that $\ell_\infty(u_k)\geq \omega_\eta$ for some $k\in \{1, \cdots, p-1\}$, then:
	\begin{equation*}
		\sum_{j=0}^{p-1}D_\infty^\circ(u_j, u_{j-1})\geq D_\infty(u, u_k)+D_\infty(u_k, v).
	\end{equation*}
	Using the lower bound in (\ref{boundD}) we get $D_\infty(u,u_k)\geq|\ell_\infty(u_k)-\ell_\infty(u)|\geq 2\omega_\eta/3$ and similarly $D_\infty(u_k,v)\geq 2\omega_\eta/3$. On the other hand, we have using the upper bound in (\ref{boundD}) that $
	D_\infty^\circ(u, v)\leq \ell_\infty(u)+\ell_\infty(v)\leq2 \omega_\eta/3$. It follows that:
	\begin{equation*}
		\sum_{j=0}^{p-1}D_\infty^\circ(u_j, u_{j-1})\geq \frac{4\omega_\eta}{3}>\frac{2\omega_\eta}{3}\geq D_\infty^\circ(u,v)\geq D_\infty(u,v).
	\end{equation*}
	This implies that, in this infimum (\ref{Dinfty}) defining $D_\infty$ we only need to consider points $u_0, u_1, \cdots, u_p$ such that $\ell_\infty(u_j)<\omega_\eta$ to compute the distance between $u$ and $v$. Since we have $D^\circ_\infty(u_j, u_{j+1})= D^{\circ, (\eta)}_\infty(u_j, u_{j+1})$ for such points
	it follows that $D^{(\eta)}_\infty(u, v)= D_\infty(u, v).$ 
	As said before, this implies that $\phi^\mathbb H$ is $\mathcal  F_\eta$-measurable for all $\eta>0$ and is therefore deterministic by a direct application of Lemma \ref{Trivialalgebra}.		
	
	To show that $\phi^\mathbb D=\phi^\mathbb H$ a.s. we then just have to show that the law of $\phi^\mathbb D$ is absolutely continuous with respect to the law of $\phi^\mathbb H$. To see this, consider the restriction of the Bessel bridge $R$ to the segment $[-1/4, 1/4]$ (where we extend $R$ by setting $R_t=R_{t+1}$ for all $t$ in  $[-1/4, 0]$), and if $\mathcal N_{\mathbb D}=\sum_{i\in I} \delta_{(t_i, \omega^i)}$ write:
	\begin{align*}
		\mathcal N_{\mathbb D}^{1/4}= \sum_{i\in I, t_i\in [0, 1/4]} \delta_{(t_i, \omega^i)} +\sum_{i\in I, t_i\in [3/4, 1]} \delta_{(t_i-1, \omega^i)}.
	\end{align*} 
	Standard results on Bessel bridges and Bessel processes show that the law of $(R_t)_{t\in [-1/4, 1/4]}$ is absolutely continuous with respect to the law of $(X_t)_{t\in [-1/4, 1/4]}$. It readily follows that the law of $((R_t)_{t\in [-1/4, 1/4]}, \mathcal N_{\mathbb D}^{1/4})$ is absolutely continuous with respect to the law of $((X_t)_{t\in [-1/4, 1/4]}, \mathcal N_{\mathbb H}^{1/4})$. Arguing similarly as in the previous paragraph, one gets that there is a measurable function $\Phi$ such that: 
	\begin{equation*}
		\begin{aligned}\phi^\mathbb D&=\Phi((R_t)_{t\in [-1/4, 1/4]}, \mathcal N_{\mathbb D}^{1/4}),\\
			\phi^\mathbb H &= \Phi((X_t)_{t\in [-1/4, 1/4]}, \mathcal N_{\mathbb H}^{1/4}).
		\end{aligned}
	\end{equation*} It follows that the law of $\phi^\mathbb D$ must be absolutely continuous with respect to the law of $\phi^\mathbb H$, the result of Proposition \ref{deterBrownPlane} follows.
\end{proof}

	\section{Some estimates for the volume of balls}
	
Recall the construction of the free Brownian disk $(\mathbb D, D)$. In this construction, the boundary $\partial \mathbb D$ is in bijection with the interval $[0,1]/_{0\sim 1}$, where the two endpoints $0$ and $1$ are identified to each other. Somewhat abusively we will write $D:[0, 1]^2\to \mathbb R_+$ for the (pseudo)-distance on $[0, 1]$ induced by the restriction of the metric $D$ on $\partial \mathbb D$ through this identification. Namely, we will simply write $D(s, t)$ to denote the quantity $D(\Gamma(s), \Gamma(t))$. For all $s\in [0, 1]$ and $r\geq 0$, we let $B_r^\mathbb D(s)=\{t\in [0, 1], D(s, t)\leq r\}$ be the ball of radius $r$ centred at $s$ for this metric.  From the construction in section $1$ and using (\ref{DistLabeling}), we see that the process $(D(0, s))_{s\in [0, 1]}$ is the (scaled) $5$-dimensional Bessel bridge $(\sqrt 3 R_t)_{t\in [0, 1]}$ of length $1$ starting and ending at $0$. It follows in particular that we have:
\begin{equation}
	\label{BallIntegralFormula}\lambda(B_{r}^{\mathbb D}(0))= \int_{0}^1\mathbbm 1_{\sqrt 3 R_s<2^{-k}}\mathrm d s,
\end{equation}
where $\lambda$ is the Lebesgue measure on $[0, 1]$.
Note that this quantity coincides with the time spent by a $5$-dimensional standard Brownian bridge inside the ball of radius $2^{-k}/\sqrt 3$. Related quantities have been heavily studied, in particular in the context of computing the Hausdorff measure of Brownian paths (see for instance \cite{6} and \cite{14}). We record a useful estimate:
\begin{proposition}
	\label{TheoBessel}
	Let $(X_t)_{t\geq 0}$ be a $5$-dimensional Bessel process started from $0$ and let $h(s)=s^2\log\log\frac{1}{s}$ for $s>0$. There exist small enough constants $\gamma>0$ and $\alpha>0$ such that for all $n\geq 1$, we have:
	\begin{equation}
		\label{BoundTheoBessel}
		\mathbb P\left(\bigcap_{k=n}^{2n} \left\{\int_0^\infty\mathbbm 1_{X_s<2^{-k}}\mathrm ds\leq \gamma h(2^{-k})\right\}\right)\leq \exp\left(-\alpha \sqrt{n}\right).
	\end{equation}
\end{proposition}
\begin{proof}[Proof of Proposition \ref{TheoBessel}]
	Following the proof of Theorem 1.2 in \cite{14}, we get the bound:
	\begin{equation*}
		\log \mathbb P\left(\bigcap_{k=n}^{m} \left\{\int_0^\infty\mathbbm 1_{X_s<2^{-k}}\mathrm ds\leq \gamma h(2^{-k})\right\}\right)\leq -\sum_{k=n}^m (k \log 2)^{-\gamma b},
	\end{equation*}
	for some constant $b>0$ and any $1\leq n\leq  m$. It follows that we can find a constant $c>0$ which might depend on $\gamma b$ but not on $m$ and $n$ such that:
	\begin{equation*}
		\log \mathbb P\left(\bigcap_{k=n}^{m} \left\{\int_0^\infty\mathbbm 1_{X_s<2^{-k}}\mathrm ds\leq \gamma h(2^{-k})\right\}\right)\leq -c ((m+1)^{1-\gamma b}-n^{1-\gamma b}).
	\end{equation*}
	In the case $m=2n$ and taking $\gamma$ small enough so that $1-\gamma b\geq 1/2$, we get the bound (\ref{BoundTheoBessel}) for a well chosen constant $\alpha$.
\end{proof}

We can use this result and formula (\ref{BallIntegralFormula}) to infer a key estimate on the volumes  $\lambda(B_r^\mathbb D(0))$ of the balls centred at $0$.
\begin{corollaire} \label{EstimatesVolumeBalls}
	Let $h(s)=s^2\log\log \frac 1 s$, and let $\alpha,\gamma$ be given as in Proposition \ref{TheoBessel}. There exists a constant $K>0$ such that for all $n\geq 1$, we have:
	\begin{equation*}
		\mathbb P\left(\bigcap_{k=n}^{2n} \left\{\lambda(B_{2^{-k}}^\mathbb D(0))\leq \frac \gamma 3 h(2^{-k})\right\}\right)\leq K\exp\left(-\alpha \sqrt n\right).
	\end{equation*}
\end{corollaire}
\begin{proof}[Proof of Corollary \ref{EstimatesVolumeBalls}]
	Using formula (\ref{BallIntegralFormula}) we get:
	\begin{equation*}
		\begin{aligned}
			\mathbb P\left(\bigcap_{k=n}^{2n} \left\{\lambda(B_{2^{-k}}^\mathbb D(0))\leq \frac \gamma 3 h(2^{-k})\right\}\right)&=	\mathbb P\left(\bigcap_{k=n}^{2n} \left\{\int_{0}^{1}\mathbbm 1_{\sqrt 3\cdot R_s<  2^{-k}}\mathrm ds\leq \frac \gamma 3 h(2^{-k})\right\}\right)\\
			&\leq \mathbb P\left(\bigcap_{k=n}^{2n} \left\{\int_{0}^{1/2}\mathbbm 1_{\sqrt 3\cdot R_s<2^{-k}}\mathrm ds\leq \frac \gamma 3 h(2^{-k})\right\}\right).
		\end{aligned}
	\end{equation*}
	As before, let $(X_s)_{s\geq 0}$ be a $5$-dimensional Bessel process started at $0$. By standard results on Bessel processes and Bessel bridges, the conditional law of $(X_s)_{s\leq 1/2}$ knowing $X_{1/2}$ is the distribution of a Bessel bridge of length $1/2$ from $0$ to $X_{1/2}$. Analogously the conditional law of $(R_s)_{s\leq 1/2}$ knowing $R_{1/2}$ is the distribution of a Bessel bridge of length $1/2$ from $0$ to $R_{1/2}$. This implies that for any bounded measurable functional $F$ on the space $\mathcal C^0([0, 1/2], \mathbb R_+)$ of all continuous functions from $[0, 1/2]$ into $\mathbb R_+$, we have:
	\begin{equation*}
		\mathbb E\left[F\left((R_t)_{0\leq t\leq 1/2}\right)\right]=\mathbb E\left[f(X_{1/2})F\left((X_t)_{0\leq t\leq 1/2}\right)\right],
	\end{equation*}
	where $f$ is the Radon--Nikodym derivative of the law of $R_{1/2}$ with respect to the law of $X_{1/2}$. Now, the laws of $X_{1/2}$ and $R_{1/2}$ are $\chi(5)$-laws scaled by some explicit constants (recall that the $\chi(5)$-distribution is the law of the norm of a standard Gaussian variable in $\mathbb R^5$). More precisely $(X_t^2)_{t\geq 0}$ (resp. $(R_{t}^2)_{t\in [0, 1]}$) is distributed as the sum of the squares of $5$ independent standard Brownian motions (resp. $5$ independent Brownian bridges). It follows that the density of $R_{1/2}$ is given by $\bar \rho(x)=\frac{32\sqrt 2}{3\sqrt \pi}x^4e^{-2x^2}\mathbbm 1_{x\geq 0}$ and the density of $X_{1/2}$ is given by $\rho(x)=\frac{8}{3\sqrt \pi}x^4e^{-x^2}\mathbbm 1_{x\geq 0}$. We therefore have $f(u)=\bar \rho(u)/\rho(u)=4\sqrt 2\exp(-u^2)$ for all $u\geq 0$. Hence for any positive measurable functional $F:\mathcal C^0([0, 1/2], \mathbb R_+)\to \mathbb R_+$ we have:
	\begin{equation*}
		\mathbb E\left[F\left((R_t)_{t\leq 1/2}\right)\right]\leq 4\sqrt 2 \ \mathbb E\left[F\left((X_t)_{t\leq 1/2}\right)\right].
	\end{equation*}
	Here, we get in particular that:
	\begin{equation}
		\label{AbsContinuity}
		\mathbb P\left(\bigcap_{k=n}^{2n} \left\{\lambda(B_{2^{-k}}^\mathbb D(0))\leq \frac \gamma 3 h(2^{-k})\right\}\right)\leq 4\sqrt 2 \ \mathbb P\left(\bigcap_{k=n}^{2n} \left\{\int_{0}^{1/2}\mathbbm 1_{\sqrt 3\cdot X_s<2^{-k}}\mathrm ds\leq \frac \gamma 3 h(2^{-k})\right\}\right).
	\end{equation}
	From scaling properties of Bessel processes, for every $c>0$, the rescaled process $(c X_{t/c^2} )_{t\geq 0}$ is also a $5$-dimensional Bessel process. It follows in particular that:
	\begin{equation*}
		\begin{aligned}
			\mathbb P\left(\bigcap_{k= n}^{2n} \left\{\int_0^{1/2}\mathbbm 1_{\sqrt 3 X_s<2^{-k}}\mathrm ds\leq \frac \gamma 3 h(2^{-k})\right\}\right)&=\mathbb P\left(\bigcap_{k= n}^{2n} \left\{\int_0^{1/2}\mathbbm 1_{X_{3s}<2^{-k}}\mathrm ds\leq \frac \gamma 3 h(2^{-k})\right\}\right)\\
			&= \mathbb P\left(\bigcap_{k=n}^{2n} \left\{\int_0^{3/2}\mathbbm 1_{X_{s}<2^{-k}}\mathrm ds\leq \gamma h(2^{-k})\right\}\right)
		\end{aligned}
	\end{equation*}
	We then use Proposition \ref{TheoBessel} to get an upper bound for the latter probability. Let $L_n=\sup\{s\geq 0 \ : \ X_s=2^{-n}\}$ and consider the event $A_n=\{L_{n}>3/2\}$. Note that on $A_n^c$ we have by construction, for all $k\geq n$,
	\begin{equation*}
		\int_0^\infty\mathbbm 1_{X_{s}<2^{-k}}\mathrm ds= \int_0^{ 3/2}\mathbbm 1_{X_{s}<2^{-k}}\mathrm ds.
	\end{equation*}By the scaling property of $X$ we easily see that $\mathbb P(A_n)=\mathbb P(L_1>3\cdot 2^{2n-1})$. Now, for instance by using the explicit density of $L_1$ in \cite{9}, we know that $\mathbb E[L_1]<\infty$ and in particular by a straightforward Markov inequality, we have $\mathbb P(A_n)\leq \frac{2}{3}\mathbb E[L_1]2^{-2n}$. Therefore:
	\begin{equation*}
		\begin{aligned}
			\mathbb P\left(\bigcap_{k=n }^{2n} \left\{\int_{0}^{3/2}\mathbbm 1_{ X_s<2^{-k}}\mathrm ds\leq \gamma h(2^{-k})\right\}\right) &\leq \mathbb P(A_n)+\mathbb P\left(A_n^c\cap\bigcap_{k=n}^{2n}\left\{\int_{0}^{3/2}\mathbbm 1_{ X_s<2^{-k}}\mathrm ds\leq \gamma h(2^{-k})\right\}\right)\\
			& \leq \frac{2}{3}\mathbb E[L_1]2^{-2 n }+\mathbb P\left(\bigcap_{k=n}^{2n} \left\{\int_{0}^{\infty}\mathbbm 1_{X_s<2^{-k}}\mathrm ds\leq \gamma h(2^{-k})\right\}\right)\\
			&  \leq \frac 2 3 \mathbb E[L_1]2^{-2n }+\exp(-\alpha\sqrt n),
		\end{aligned}
	\end{equation*}
	by Proposition \ref{TheoBessel}.
	Using this bound and (\ref{AbsContinuity}), we get the conclusion of the proposition for some well chosen constant $K$.
\end{proof}
	
	\section{Upper bound for the Hausdorff measure}
	
We now proceed to the proof of Theorem \ref{MainResult}. The main idea is to use Lemma \ref{Hausdorffmeasurescontrole} to verify that $m^h_{\mathbb D}$ and the Lebesgue measure $\lambda$ are mutually absolutely continuous and then to apply Proposition \ref{deterBrownPlane} to show that the corresponding Radon--Nikodym derivative is constant almost surely. We first derive a useful bound for the distance $D$:
\begin{lemma}
	For almost all $\omega\in \Omega$, we can find a constant $c(\omega)$, which depends on $\omega$, such that for every $s,t\in [0, 1]$:
	\begin{equation}
		\label{MajoDistance}
		D(s, t)\leq c(\omega)\left(1+\log\frac{1}{|t-s|}\right)|t-s|^{1/2}.
	\end{equation}
\end{lemma}
\begin{proof}[Proof]
	Using the rerooting property of Lemma \ref{rerooting} we have for all $n\geq 1$ and all $i\in \{1, \cdots, 2^n\}$:
	\begin{equation*}
		\mathbb P\left(D((i-1)2^{-n}, i2^{-n})>n2^{-n/2}\right)=\mathbb P\left(\sqrt 3\cdot R_{2^{-n}}>n2^{-n/2}\right)\leq \tilde c e^{-n},
	\end{equation*}
	for some constant $\tilde c>0$. To derive the last bound, we use the fact that $R_{2^{-n}}$ has the same law as the norm of a centred Gaussian vector with covariance matrix $2^{-n}(1-2^{-n})\text{Id}$ in $\mathbb R^5$. It follows by the Borel--Cantelli lemma that almost surely we can find a (random) integer $n_0$ such that:
	\begin{equation*}
		\forall n\geq n_0, \forall i\in \{1, \cdots, 2^n\}, \ \  D((i-1)2^{-n}, i2^{-n})\leq n2^{-n/2}.
	\end{equation*}
	By standard chaining arguments, it follows that almost surely:
	\begin{equation*}
		\sup_{s, t\in [0, 1[, s\neq t} \frac{D(s, t)}{\left(1+\log(1/|t-s|)\right)|t-s|^{1/2}}<\infty,
	\end{equation*}
	which gives the conclusion of the lemma.
\end{proof}
\begin{proposition}
	Let $\gamma$ be as in Proposition \ref{EstimatesVolumeBalls}. We have almost surely that:
	\begin{equation*}
		m^h_\mathbb D\left(\left\{t\in [0, 1] \ : \ \limsup_{r\downarrow 0}\frac{ \lambda(B_r^\mathbb D(t ))}{h(r)}<\frac{\gamma}{12}\right\}\right)=0.
	\end{equation*}
\end{proposition}
\begin{proof}[Proof]
	It suffices to show that, for every $n_0\geq 1$, the  $m^h_\mathbb D$-measure of:
	\begin{equation*}
		\mathcal B_{n_0}:=\left\{t\in[0, 1] : \lambda(B_{2^{-k}}^\mathbb D(t))\leq \frac{\gamma}{12} h(2^{-k}), \forall k\geq n_0 \right\}, 
	\end{equation*}
	is almost surely $0$. Let us introduce:
	\begin{equation*}
		\mathcal B_{n_0, n}:=\left\{t\in[0, 1] : \lambda(B_{2^{-k}}^\mathbb D(t))\leq \frac \gamma 3 h(2^{-k}), \forall k\in\{n_0, n_0+1, \cdots, n\} \right\}.
	\end{equation*}
	\underline{Claim:} For all $p$ large enough, depending on $\omega$, we have for all $t\in [0, 1]$:
	\begin{equation*}
		t\in \mathcal B_{n_0}\Rightarrow 2^{-4p}\lfloor 2^{4p} t\rfloor\in \mathcal B_{n_0+1, p}.
	\end{equation*}  
	Let us prove this claim. Write  $\lfloor t\rfloor_p=2^{-4p}\lfloor 2^{4p} t\rfloor$ to simplify notation. For $p$ large enough, we have using (\ref{MajoDistance}):
	\begin{equation*}
		D(\lfloor t\rfloor_p, t)\leq c(\omega)\left(1+\log 2^{4p}\right)2^{-2p}\leq 2^{-p-1}.
	\end{equation*}
	In particular, it readily follows that for all $k\leq p$, we have $B_{2^{-k-1}}^\mathbb D(\lfloor t\rfloor_p)\subset B_{2^{-k}}^\mathbb D(t)$. Now suppose that $t\in \mathcal B_{n_0}$, then for all $n_0\leq k\leq p$:
	\begin{equation*}
		\lambda(B_{2^{-k-1}}^\mathbb D(\lfloor t\rfloor_p))\leq \lambda(B_{2^{-k}}^{\mathbb D}(t))\leq \frac{\gamma}{12}h(2^{-k})\leq \frac \gamma 3 h(2^{-k-1}),
	\end{equation*}
	which implies that $\lfloor t\rfloor_p\in \mathcal B_{n_0+1, p}$. This completes the proof of our claim.
	
	In particular, the claim shows that when $p$ is large enough:
	\begin{equation*}
		\mathcal B_{n_0}\subset  \bigcup_{\substack{i\in \{0, \cdots, 2^{4p}-1\}\\ i2^{-p}\in \mathcal B_{n_0+1, p}}}\left[i2^{-4p}, (i+1)2^{-4p}\right],
	\end{equation*}
	and the $D$-diameter of each of the intervals in the union of the last display is at most $c(\omega)(1+\log2^{4p})2^{-2p}$. It follows that:
	\begin{equation}
		\label{ineqmh}
		m^h_\mathbb D(\mathcal B_{n_0})\leq \liminf_{p\to \infty} \left[\text{Card}\left\{i\in \{0,\cdots, 2^{4p}-1\}: \ \frac{i}{2^{4p}}\in \mathcal B_{n_0+1, p} \right\} h\left(c(\omega)\frac{1+\log 2^{4p}}{2^{2p}}\right)\right].
	\end{equation}
	Using Corollary \ref{EstimatesVolumeBalls} and the rerooting property of Lemma \ref{rerooting}, we see that whenever $p\geq 2(n_0+1)$:
	\begin{equation*}
		\mathbb E\left[\text{Card}\left\{i\in \{0,\cdots, 2^{4p}-1\}: \ i2^{-4p}\in \mathcal B_{n_0+1, p} \right\}\right]=2^{4p}\mathbb P(0\in \mathcal B_{n_0+1, p})\leq 2^{4p} K \exp\left(-\alpha\sqrt{\lfloor p/2\rfloor}\right).
	\end{equation*}
	Using a Markov inequality and the Borel--Cantelli lemma, it follows that for any $\beta\in (0, \alpha/\sqrt 2)$, we have a.s.:
	\begin{equation*}
		2^{-4p}\exp(\beta\sqrt p) \text{Card}\{i\in \{0, \cdots, 2^{4p}-1\}: \ i2^{-4p}\in \mathcal B_{n_0+1, p}\}\xrightarrow[p\to \infty]{} 0.
	\end{equation*}
	Since $h\left(c(\omega)(1+\log 2^{4p})2^{-2p}\right)<2^{-4p}\exp(\beta\sqrt p)$ as soon as $p$ is large enough, it follows that:
	\begin{equation*}
		\text{Card}\left\{i\in \{0,\cdots, 2^{4p}-1\}: \  i2^{-4p}\in \mathcal B_{n_0+1, p} \right\} \times h\left(c(\omega)(1+\log 2^{4p})2^{-2p}\right)\xrightarrow[p\to\infty]{} 0.
	\end{equation*}
	And (\ref{ineqmh}) gives the desired result.
\end{proof}
\begin{corollaire}
	There exist constants $\kappa_1, \kappa_2>0$ such that for any Borel set $A\subset [0, 1]$, we have:
	\begin{equation*}
		\kappa_2\lambda(A) \leq m^h_\mathbb D(A)\leq \kappa_1 \lambda(A).
	\end{equation*}
\end{corollaire}
\begin{proof}[Proof]
	Let $A$ be a Borel set in $[0, 1]$. For the upper bound, consider:
	\begin{equation*} 
		A'= A\setminus \left\{t\in [0, 1]\ : \ \limsup_{r\downarrow 0^+}\frac{ \lambda(B_r^\mathbb D(t ))}{h(r)}<\frac{\gamma}{12}\right\}.
	\end{equation*} 
	Using the previous proposition we have for every $x\in A'$ that $m^h_\mathbb D( A')= m^h_\mathbb D(A)$. Moreover by construction, we have $\limsup_{r\downarrow 0^+}\frac{ \lambda(B_r^\mathbb D(t ))}{h(r)}\geq \frac{\gamma}{12}$. By a direct application of Lemma \ref{Hausdorffmeasurescontrole}, (ii), there exists a constant $\kappa_2$ independent of $A$ such that $m^h_\mathbb D(A)=m^h_\mathbb D(A')\leq \kappa_2\lambda(A')\leq \kappa_2\lambda(A)$.\\
	
	For the lower bound, note that there exists a constant $c>0$ such that \begin{equation*}
		\limsup_{r\to 0^+}\ h(r)^{-1}\int_0^1 \mathbbm 1_{\sqrt 3\cdot R_s<r} \mathrm ds\leq c.
	\end{equation*} To see this, note that this result is well known when one replaces $R$ by the Bessel process $X$, see for instance Theorem $3$ in \cite{6} which provides an explicit value for this $\limsup$. Then we just have to use the absolute continuity argument of  the proof of Corollary \ref{EstimatesVolumeBalls}. The last display shows that almost surely $\limsup_{r\downarrow 0}\frac{\lambda(B_r^\mathbb D(0))}{h(r)}\leq c$ and by the rerooting property of Lemma \ref{rerooting}, it follows that almost surely we have:
	\begin{equation*}
		\lambda\left(\left\{x\in [0, 1] \ : \ \limsup_{r\to \infty}\frac{\lambda(B_r^\mathbb D(x))}{h(r)}>c\right\}\right)=0.
	\end{equation*}  Now let:
	\begin{equation*}
		A''=A\setminus\left\{x\in [0, 1]\ : \ \limsup_{r\to \infty}\frac{\lambda(B_r^\mathbb D(x))}{h(r)}>c\right\},
	\end{equation*}
	we have $\lambda(A'')=\lambda(A)$ and $\limsup_{r\downarrow0} \frac{\lambda (B_r^{\mathbb D}(x))}{h(r)}\leq c$ for every $x\in A''$. By a direct application of Lemma \ref{Hausdorffmeasurescontrole}, (i) it follows that $\lambda(A)=\lambda(A'')\leq \kappa_1m^h_\mathbb D(A'')\leq \kappa_1m^h_\mathbb D(A)$ with a constant $\kappa_1>0$ independent of $A$.
\end{proof}
\noindent We finally use Proposition \ref{deterBrownPlane} to show that $m^h_{\mathbb D}=\kappa \lambda$ for some deterministic $\kappa>0$. 
\begin{theorem}
	There exists a constant $\kappa\in[\kappa_1, \kappa_2]$ such that $m^h_\mathbb D=\kappa\lambda$ almost surely.
\end{theorem}
\begin{proof}[Proof]
	By the previous proposition we see that the measures $\lambda$ and $m^h_\mathbb D$ are mutually absolutely continuous, with Radon Nikodym derivative $\frac{\mathrm dm^h_\mathbb D}{\mathrm d\lambda}$ taking values in $[\kappa_1, \kappa_2]$. Let us write $\psi(s)=m^h_\mathbb D([0, s])$. Standard results on cumulative distributive functions show that $\psi$ is differentiable $\lambda$-a.e. and that $\psi'(s)=\frac{\mathrm dm^h_\mathbb D}{\mathrm d\lambda}(s)$ $\lambda$-a.e.. From the rerooting property of Lemma \ref{rerooting}, we get that  $\psi'(x)$ and $\psi'(0)$ have the same law for every $x\in [0, 1]$. But we have a.s.
	\begin{equation*}
		\psi'(0)=\limsup_{\epsilon\downarrow 0^+} \frac{m^h_\mathbb D([0, \epsilon])}{\epsilon}=\kappa,
	\end{equation*}
	where $\kappa$ is the constant introduced in Proposition \ref{deterBrownPlane}. This means that, almost surely, $\lambda$-almost everywhere, $\frac{\mathrm dm^h_\mathbb D}{\mathrm d\lambda}=\psi'=\kappa$ with $\kappa \in [\kappa_1, \kappa_2]$. The result of the proposition follows.
\end{proof}

\section*{Acknowledgements}
I would like to express my sincere thanks to Jean-François Le Gall for all his suggestions and support throughout the writing of this work, and for his meticulous reading of the ---numerous--- previous versions of this document.

\end{document}